\documentclass[leqno]{amsart}
\usepackage{amsfonts,amssymb,amsmath,amsgen,amsthm}
\usepackage{pdfsync}

\usepackage{relsize} 
\usepackage{bigints}
\usepackage{hyperref,color}

\theoremstyle{plain}
\newtheorem{theorem}{Theorem}[section]
\newtheorem{definition}[theorem]{Definition}
\newtheorem{lemma}[theorem]{Lemma}
\newtheorem{corollary}[theorem]{Corollary}
\newtheorem{proposition}[theorem]{Proposition}

\theoremstyle{remark}

\def\d{{\partial}}
\def\R{{\mathbb R}}
\def\T{{\mathbb T}}
\def\N{{\mathbb N}}
\def\Z{{\mathbb Z}}

\def\P{\mathbf P}
\def\Q{\mathbf Q}
\def\dd{\mathrm d}
\def\eps{\varepsilon}
\newcommand{\seps}{\sigma_{\eps}}
\newcommand{\tseps}{\tilde{\sigma}_{\eps}}
\newcommand{\reps}{\rho_{\eps}}

\newcommand{\ueps}{u_{\eps}}
\newcommand{\phie}{\phi_{\eps}}
\newcommand{\dive}{\mathop{\mathrm{div}}}
\newcommand{\Feps}{F_{\eps}}

\newcommand{\Heps}{H_{\eps}}
\newcommand{\eith}{e^{it\Heps}}
\newcommand{\Ueps}{U_{\eps}}
\newcommand{\Jeps}{J_{\eps}}
\newcommand{\tJeps}{\tilde{J}_{\eps}}
\def\P{\mathbf P}
\def\Q{\mathbf Q}
\newcommand{\supp}{\operatorname{supp}}

\makeindex             

\date\today

\title[Acoustic oscillations in systems for quantum fluids]{Analysis of acoustic oscillations for a class of hydrodynamic systems describing quantum fluids}
\author[P. Antonelli]{Paolo Antonelli}
\address{
Gran Sasso Science Institute, viale Francesco Crispi, 7, 67100 L'Aquila, Italy}
\email{paolo.antonelli@gssi.it}

\author[L.E. Hientzsch]{Lars Eric Hientzsch}
\address{
Univ. Grenoble Alpes, CNRS, Institut Fourier, 100 rue des Math{\'e}matiques, 38610 Gi{\`e}res, France}
\email{larseric.hientzsch@univ-grenoble-alpes.fr}

\author[P. Marcati]{Pierangelo Marcati}
\address{
Gran Sasso Science Institute, viale Francesco Crispi, 7, 67100 L'Aquila, Italy}
\email{pierangelo.marcati@gssi.it}

\begin{document}

\begin{abstract}Hydrodynamic systems for quantum fluids are systems for compressible fluid flows for which quantum effects are macroscopically relevant. We discuss how the presence of the dispersive tensor describing the quantum effects alters the acoustic dispersion at the example of the Quantum Hydrodynamic system (QHD). For the QHD system the dispersion relation is given by the Bogoliubov dispersion relation for weakly interacting Bose gases. We provide refined Strichartz estimates allowing for an accurate control of acoustic oscillations. Applications to the low Mach number limit for the quantum Navier-Stokes equations and the QHD system are discussed. The dispersive analysis generalizes to some Euler- and Navier-Stokes-Korteweg systems for capillary fluids.
\end{abstract}

\maketitle

\section{Introduction}
Acoustic oscillations are significant for the behavior of solutions to hydrodynamic systems and their control is required for e.g. the Cauchy Theory, large time behavior and singular limits. For instance, it is well-known that fast acoustic oscillations occur in the low Mach number regime \cite{FN}. To obtain a suitable control, one may exploit the dispersion of acoustic waves. Dispersive phenomena highly depend on the domain, here we consider $\R^d$ for $d\geq 2$. The main objective of these notes is twofold: first, to examine how the presence of quantum effects alters the acoustic dispersion and elucidate its link to the Bogoliubov excitation spectrum \cite{B} for weakly interacting Bose gases. Second, we provide refined Strichartz estimates that allow for an accurate control of the acoustic dispersion. These refined estimates are crucial for the study of the incompressible limit in the class of weak solutions for general ill-prepared data \cite{AHM}. The prototype model for quantum fluids, i.e. compressible, inviscid fluids subject to quantum (dispersive) effects is the  quantum hydrodynamic system (QHD) \cite{AHMZ, AM, AM16} that reads
\begin{equation}\label{eq:QHDintro}
\begin{aligned}
&\d_t\rho+\dive J=0\\
&\d_tJ+\dive\left(\frac{J\otimes J}{\rho}\right)+\nabla p(\rho)=\frac{1}{2}\rho\nabla\left(\frac{\Delta\sqrt{\rho}}{\sqrt{\rho}}\right).
\end{aligned}
\end{equation}
The dynamics is posed on $\R^d$ for $d\geq 2$ and equipped with non-vanishing boundary conditions at infinity
\begin{equation}\label{eq:boundaryQHDintro}
    \rho(x)\rightarrow 1, \qquad |x|\rightarrow \infty.
\end{equation}
The unknowns are the mass density $\rho$ and the current density $J$, the pressure is denoted by $p(\rho)$ and given by the barotropic $\gamma$-law $p(\rho)=\rho^{\gamma}$ with $\gamma>1$. The nonlinear third-order dispersive tensor on the right-hand side accounts for the quantum effects of fluid. Beyond superfluidity \cite{GP} and Bose-Einstein condensation (BEC) \cite{PS}, the QHD and related systems appear in superconductivity, quantum plasmas and the modelling of semi conductor devices, see \cite{AM16}, also \cite{J} for related viscous models. In a more general framework, the QHD system belongs to the class of Euler-Korteweg fluids describing inviscid capillary fluids. The dispersive analysis described here generalizes to certain Euler / Navier-Stokes-Kortweg systems \cite{BG,BDL,BGL}, e.g. to shallow water waves. Our purpose is to discuss the behavior of perturbations of the constant solution $(\rho=1,J=0)$ for which the non-trivial far-field behavior \eqref{eq:boundaryQHDintro} provides the suitable framework. Equipped with \eqref{eq:boundaryQHDintro}, system \eqref{eq:QHDintro} arises in application to BEC and superfluidity close to the $\lambda$-point \cite{PS} and exhibits a variety of special solutions. We notice that system \eqref{eq:QHDintro} is Hamiltonian, whose energy
\begin{equation}\label{eq:en_QHDintro}
\mathcal{E}(\rho,J)=\int\frac{1}{2}|\nabla\sqrt{\rho}|^2+\frac12\frac{|J|^2}{\rho}+\pi(\rho) \dd x,
\end{equation}
is formally conserved along the flow of solutions. The internal energy is given by
\begin{equation*}
    \pi(\rho)=\frac{\rho^{\gamma}-\gamma(\rho-1)}{\gamma(\gamma-1)},
\end{equation*}
and encodes the far-field behavior \eqref{eq:boundaryQHDintro}.
From a mathematical point of view, system \eqref{eq:QHDintro} is a compressible Euler system augmented by a nonlinear stress tensor encoding the quantum effects
\begin{equation}\label{eq:Bohmintro}
\frac{1}{2}\rho\nabla\left(\frac{\Delta\sqrt{\rho}}{\sqrt{\rho}}\right)=\frac{1}{4}\nabla\Delta\rho- \dive(\nabla\sqrt{\rho}\otimes\nabla\sqrt{\rho})=\frac{1}{4}\dive(\rho\nabla^2\log\rho),
\end{equation}
under suitable regularity assumptions.
The stress tensor is commonly referred to as Bohm potential or quantum pressure. A first heuristic argument to elucidate the difference between \eqref{eq:QHDintro} and the compressible Euler equations consists in determining the length scales on which \eqref{eq:Bohmintro} is relevant. System \eqref{eq:QHDintro} is given in dimensionless form. If $L$ denotes the characteristic length scale for variations of the density $\rho$, then comparing the order of magnitude yields
\begin{equation*}
    p(\rho)\sim \rho^{\gamma}, \qquad p_{quantum}=\frac{1}{2}\rho\nabla^2\log\rho\sim \frac{1}{2L^2},
\end{equation*}
Hence, the quantum pressure is relevant up to length-scales of order $1$, while it becomes negligible for large length scales. In the context of BEC and superfluidity, the characteristic length scale $L$ is given by the healing length $\eps$. It determines the core size of quantum vortices ($d=2$) being proportional to $\eps$, \cite{P}. The suitable asymptotic regime to study vortices is introduced via the parabolic scaling
\begin{equation*}
    t'=\frac{t}{\eps^2}, \qquad x'=\frac{x}{\eps}, \qquad J'(t',x')=\frac{1}{\eps}J(\frac{t}{\eps^2},\frac{x}{\eps}),
\end{equation*}
in which the vortex core is of infinitesimal size $O(\eps)$, and density variations occur over distances of $O(\eps)$, the quantum pressure is relevant for length scales of $O(\eps)$. Applying this scaling to \eqref{eq:QHDintro}, we obtain
\begin{equation}\label{eq:scaledQHD}
\begin{aligned}
&\d_t\reps+\dive \Jeps=0\\
&\d_t \Jeps+\dive\left(\frac{\Jeps\otimes \Jeps}{\reps}\right)+\frac{1}{\eps^2}\nabla p(\reps)=\frac{1}{2}\reps\nabla\left(\frac{\Delta\sqrt{\reps}}{\sqrt{\reps}}\right).
\end{aligned}
\end{equation}
Hence, the $\eps$-limit corresponds to a low Mach number regime. In order to study acoustic oscillations, we linearize system \eqref{eq:scaledQHD} around the constant solution $(\rho,J)=(1,0)$. We consider the density fluctuations $\seps:=\eps^{-1}(\reps-1)$ and obtain
\begin{equation}\label{eq:linear}
\begin{aligned}
&\partial_t\seps+\frac{1}{\eps}\dive \Jeps=0,\\
&\partial_t \Jeps+\frac{1}{\eps}\nabla\seps-\frac{\eps}{4}\nabla\Delta \seps=F_{\eps},
\end{aligned}
\end{equation}
where we have used \eqref{eq:Bohmintro} and denote
\begin{equation*}
    F_{\eps}=-(\gamma-1)\nabla\pi_{\eps}(\reps)-\dive\left(\frac{\Jeps\otimes\Jeps}{\reps}\right)+\dive\left(\nabla\sqrt{\reps}\otimes\nabla\sqrt{\reps}\right).
\end{equation*}
The remainder of this paper is dedicated to the dispersive analysis of system \eqref{eq:linear} for small $\eps>0$. Section \ref{sec:linearized} derives the augmented dispersion relation and discusses its physical motivations. 
The refined Strichartz estimates are introduced in Section \ref{sec:dispersive}. Finally, we discuss applications in Section \ref{sec:app}.

\section{Augmented dispersion relation}\label{sec:linearized}
Projecting the second equation on irrotational fields by means of the Leray-Helmholtz projections, defined by $\Q:=\nabla\Delta^{-1}\dive$ and $\P:=\mathbb{Id}-\Q$ respectively, we obtain from \eqref{eq:linear} the system describing the acoustic waves 
\begin{equation}\label{eq:linearQ}
\begin{aligned}
&\partial_t\seps+\frac{1}{\eps}\dive \Q(\Jeps)=0,\\
&\partial_t \Q(\Jeps)+\frac{1}{\eps}\nabla\seps-\frac{\eps}{4}\nabla\Delta \seps=\Q(F_{\eps}).
\end{aligned}
\end{equation}
Formally, the density fluctuations satisfy the Boussinesq type equation
\begin{equation}\label{eq:densityfluct}
    \partial_{tt}^2\seps-\frac{1}{\eps^2}\Delta(1-\frac{\eps^2}{4}\Delta)\seps=-\dive\Q(\Feps).
\end{equation}
By looking for plane wave solutions of the type $e^{i\xi\cdot x+i\omega t}$ to the homogeneous equations, one recovers the dispersion relation
\begin{equation}\label{eq:dispersionrelation}
    \omega=\frac{|\xi|}{\eps}\sqrt{1+\left(\frac{\eps}{2}|\xi|\right)^2}. 
\end{equation}
We notice that the dispersion relation behaves like $\omega\sim \frac{|\xi|}{\eps}$ for frequencies below the threshold $\frac{1}{\eps}$ while it behaves like $\omega\sim |\xi|^2$ for frequencies above the threshold. The dispersion relation is hence wave-like for low frequencies and Schr{\"o}dinger like for high frequencies. This distinguishes the dispersion of acoustic oscillations from the one for classical compressible fluids where the acoustic dispersion is governed by the wave equation with speed $\frac{1}{\eps}$. The augmented dispersion relation allows one to infer better dispersive estimates leading to a faster decay of acoustic oscillations, see Corollary \ref{coro:dispersive}.
Before proceeding to the mathematical analysis in Section \ref{sec:dispersive}, we provide a brief physical motivation for the augmented dispersion.

System \eqref{eq:QHDintro} arises as model in a variety of contexts mentioned in the introduction. Here, we focus on the derivation of the QHD model for Bose-Einstein condensates and discuss how the dispersion relation \eqref{eq:dispersionrelation} is linked to the Bogoliubov dispersion law \cite{B} describing the elementary excitations for weakly interacting Bose gases. In 1947, Bogoliubov  introduces the first microscopic theory for weakly interacting Bose gases. In \cite{B}, it is shown that the quantum many body Hamiltonian describing the Bose gas can be approximated by a Hamiltonian of independent quasi particles (Bogoliubov approximation). The Hamiltonian of quasi-particles turns out to be diagonalizable and reads to leading order
\begin{equation}\label{eq:HM}
    H=E_0+\sum_p c(p)\hat{b}^\dagger\hat{b},
\end{equation}
where $\hat{b}^{\dagger}$ and $\hat{b}$ denote creation and annihilation operators for quasi-particles respectively. 
Bogoliubov derives the ground state energy $E_0$ and the spectrum of elementary excitations $c(p)$ described by the Bogoliubov dispersion law,
\begin{equation}\label{eq:Bmom}
    c(p)=\sqrt{\frac{gn}{m}p^2+\left(\frac{p^2}{2m}\right)^2}.
\end{equation}
The excitation spectrum is phonon-like for small momenta and of free particle nature for large momenta. Moreover, \eqref{eq:Bmom} allowed to establish a link between BEC and superfluidity on microscopic level through Landau's criterion of superfluidity \cite{PS}. The validity of \eqref{eq:Bmom} has only recently been rigorously proven, see \cite{BBCS, S}. The Gross-Pitaevskii theory \cite{P, PS} introduced a macroscopic description of the weakly interacting Bose gas in terms of an effective wave-function, also called order-parameter, satisfying the adimensionalized Gross-Pitaevskii equation,
\begin{equation}\label{eq:GP}
    i\partial_t\psi=-\frac{1}{2}\Delta\psi+(|\psi|^2-1)\psi,
\end{equation}
where the characteristic length scale is given in units of the healing length $\eps=\frac{\hbar}{\sqrt{2}mc}$. The parameters $\hbar, m, c $ denote respectively, the reduced Planck constant, the mass and the sound speed. Physically, the healing length $\eps$ describes for a perturbed condensate the distance over which the density reaches its unperturbed uniform value. The Madelung transform establishes the formal equivalence of \eqref{eq:GP} and \eqref{eq:QHDintro}, see e.g. \cite{AHMZ}. The dispersion relation \eqref{eq:Bmom} arises naturally in the hydrodynamic framework of the QHD system, see \eqref{eq:dispersionrelation} with $\eps=1$ and \eqref{eq:linearQ} and is recovered for \eqref{eq:GP} by looking for plane wave solutions to a suitable linearization of \eqref{eq:GP}.\\
To conclude this section, we mention that \eqref{eq:QHDintro} can be considered in the framework of Korteweg fluids. The presence of a capillary tensor leads to an augmented dispersion relation. The density fluctuations for a fairly general class of Korteweg fluids \cite{BG, BGL} are still governed by the Boussinesq type equation \eqref{eq:densityfluct}.

\section{Dispersive analysis}\label{sec:dispersive}
In this section, we prove refined Strichartz estimates for a symmetrization of the linearized system \eqref{eq:linear}. Our analysis can be considered as the $\eps$-dependent version of the dispersive analysis in \cite{GNT05, GNT09} for the linearization of \eqref{eq:GP}. We stress that the non-homogeneity of the dispersion relation does not allow for a separation of scales and thus the $\eps$-dependent version is not direct consequence of a scaling argument as opposed to classical fluid dynamics. System \eqref{eq:linear} has been studied in \cite{BDS10} in the context of the linear wave regime for \eqref{eq:GP}. In \cite{BDS10}, the acoustic decay is achieved by considering two different regimes for low and high frequencies. This frequency splitting entails a loss of derivatives. Proposition \ref{prop:Strichartz_A} and Proposition \ref{prop:Strichartz_B} improve the estimates of \cite{BDS10} as decay of acoustic waves with arbitrarily small loss of regularity is achieved. 
This paragraph is based on \cite{AHM} for $d\geq 3$. We complement the analysis of \cite{AHM} by providing an estimate for $d=2$.\\
We introduce the parameter $\kappa$ to emphasize the contribution of the dispersive tensor to the dispersion relation and to provide more general estimate directly applicable in other contexts such as Euler-/ Navier-Stokes-Korteweg fluids \cite{BG, BDL,BGL} where $\kappa$ depends on the linearization of the capillarity tensor. To recover \eqref{eq:scaledQHD}, one sets $\kappa=\frac{1}{2}$. We symmetrize \eqref{eq:linear} by introducing the change of variables
\begin{equation}\label{eq:defitilde}
\tseps:=(1-\eps^2\kappa^2\Delta)^{\frac{1}{2}} \seps, \hspace{1.5cm} \tJeps:=(-\Delta)^{-\frac{1}{2}}\dive\Jeps,
\end{equation}
and check that if $(\seps,\Q(\Jeps))$ is a solution of \eqref{eq:linear} then $(\tseps,\tJeps)$ satisfies the symmetrized system
\begin{equation}\label{eq:symAW}
\begin{aligned}
&\partial_t\tseps+\frac{1}{\eps}(-\Delta)^{\frac{1}{2}}(1-\kappa^2\eps^2\Delta)^{\frac{1}{2}}\tJeps=0,\\
&\partial_t\tJeps-\frac{1}{\eps}(-\Delta)^{\frac{1}{2}}(1-\kappa^2\eps^2\Delta)^{\frac{1}{2}}\tseps=\tilde{F}_{\eps},
\end{aligned}
\end{equation}
where $\tilde{F}_{\eps}=(-\Delta)^{-\frac{1}{2}}\dive F_{\eps}$. The bounds on $(\tseps, \tJeps)$ will then allow for the desired control of $(\seps,\Q(\Jeps))$, see Corollary \ref{coro:3d}. We introduce the operator $\Heps=\frac{1}{\eps}\sqrt{(-\Delta)(1-(\eps\kappa)^2\Delta)}$ being defined as Fourier multiplier with symbol $\phie(|\xi|)=\kappa^{-1}\omega(\kappa|\xi|)$ with $\omega$ as in \eqref{eq:dispersionrelation}. System \eqref{eq:symAW} is characterized by the unitary semigroup operator $\eith$. Denoting $\phi(r)=r\sqrt{1+\kappa^2r^2}$, we observe that $\phie$ has the scaling property
\begin{equation}\label{eq:phiescaled}
    \phie(|\xi|)=\frac{1}{\eps^2}\phi(\eps|\xi|),
\end{equation}
that allows us to infer the $\eps$-dependent estimate from the ones available for $e^{itH_1}$, namely for $\eps=1$. The respective operator has been studied in \cite{GNT05, GNT09}. We rely on Theorem 2.2 in \cite{GNT05} providing the following stationary phase estimate. The estimate is known to be sharp \cite{COX}.
\begin{lemma}[\cite{GNT05}]\label{lem:GNT}
Let $\phi(r)\in C^{\infty}(0,\infty)$ be such that 
\begin{enumerate}
\item $\phi'(r),\phi''(r)>0 $ for all $r>0$.
\item $\phi'(r)\sim \phi'(s)$ and $\phi''(r)\sim \phi''(s)$ for all $0<s<r<2s$.
\item $\left|\phi^{(k+1)}(r)\right|\lesssim \frac{\phi'(r)}{r^k}$ for all $r>0$ and $k\in \N$.
\end{enumerate}
Let $\chi(r)$ be a dyadic cut-off function with support around $r \sim R$ satisfying
\begin{equation*}
|\chi^{(k+1)}(r)|\leq C R^{-k},
\end{equation*}
for an absolute constant $C=C(k)>0$ independent from $r$ and $R$. Then, for $d
\geq 2$ if
\begin{equation*}
I_\phi(t, x, R):=\int_{\R^d}e^{it\phi(|\xi|)+ix\cdot\xi}\chi(|\xi|)\,\dd \xi
\end{equation*}
we have
\begin{equation}\label{eq:estGNT}
\sup_{x\in\R^d}\left|I_{\phi}(t,x,R)\right|\lesssim t^{-\frac{d}{2}}\left(\frac{\phi'(R)}{R}\right)^{-\frac{d-1}{2}}\left(\phi''(R)\right)^{-\frac{1}{2}}.
\end{equation}
\end{lemma}
We notice that as $\phi$ is radial one has
\begin{equation}\label{eq:h}
    h(r):=\det(Hess(\phi(r))=\left(\frac{\phi'(r)}{r}\right)^{d-1}\phi''(r),
\end{equation}
so that one recasts $h(r)^{-\frac{1}{2}}$ on the right-hand side of \eqref{eq:estGNT}.
\begin{lemma}\label{lem:B3}
Let $h$ be defined as in \eqref{eq:h}.
There exists $C>0$ such that for any $r\in [0,\infty]$,
\begin{equation}\label{eq:boundh}
    0\leq h(r)^{-\frac{1}{2}}\leq C \frac{1}{\kappa^{\frac{d}{2}}}\left( \frac{\kappa r}{\sqrt{1+(\kappa r)^2}}\right)^{\frac{d-2}{2}}.
\end{equation}
For $d=2$, there exists $C>0$ such that for any $r \in [0,\infty)$,
    \begin{equation*}
     \frac{1}{C}\kappa^{-1}\leq h(r)^{-\frac{1}{2}}\leq C\kappa^{-1}.
    \end{equation*}
\end{lemma}
Estimate \eqref{eq:h} shows that $h(r)^{-\frac{1}{2}}$ is uniformly bounded from above. Hence, \eqref{eq:h} implies that the propagator $e^{itH_1}$ satisfies in particular the dispersive estimate of the Schr{\"o}dinger semigroup operator $e^{it\Delta}$ provided that $\kappa>0$ is fixed. Again for $\kappa>0$ fixed, we observe that \eqref{eq:h} yields a regularizing effect at low frequencies for $d>2$. Indeed, 
\begin{equation}\label{eq:reg}
    h(r)^{-\frac{1}{2}}\leq C\kappa^{-1}r^{\delta},
\end{equation}
for any $\delta\in[0,\frac{d-2}{2}]$. This regularizing effect has been observed in \cite{GNT05} and subsequently exploited in \cite{GNT09} for the scattering properties of \eqref{eq:GP}. The different behavior depending on the dimension can be explained by the curvature of the hypersurface $|\xi|\sqrt{1+|\xi|^2}$ w.r.t. the hypersurface $|\xi|^2$.
Denoting $h_{\eps}(r):=\det(Hess(\phie(r))$, one has that $h_{\eps}(r)=h(\eps r)$ for $r>0$ in view of \eqref{eq:phiescaled}.
Exploiting \eqref{eq:phiescaled} and \eqref{eq:boundh}, we derive the respective dispersive estimate for $e^{it\Heps}$.
\begin{corollary}\label{coro:dispersive}
Let $d\geq 2$, $\phie$ as in \eqref{eq:phiescaled}, $R>0$ and let $\chi(r)\in C_c(0,\infty)$ be as in Lemma \ref{lem:GNT}. Then there exists a constant $C>0$ such that
\begin{equation}\label{eq:disp1}
\sup_{x\in \R^d}\left|\int_{\R^d}e^{it\phie(|\xi|)+ix\cdot\xi}\chi(|\xi|)\,\dd \xi\right|\leq C \frac{C}{\kappa^{\frac{d}{2}}}t^{-\frac{d}{2}}\left(\frac{\kappa\eps R}{\sqrt{1+(\eps\kappa)^2R^2}}\right)^{\delta},
\end{equation}
for any $\delta\in[0,\frac{d-2}{2}]$.
\end{corollary}
For $\delta=0$, estimate \eqref{eq:disp1} yields the classical dispersive estimate for the Schr{\"o}dinger operator. We recover the mentioned regularizing effect for low frequencies due to \eqref{eq:reg} in the $\eps$-dependent estimate \eqref{eq:disp1}. More precisely, one may trade a loss of regularity $R^{\delta}$ in order to gain a factor $\eps^{\delta}$ provided that $\kappa$ is fixed.

\begin{proof}
Introducing the variables $t'=\frac{t}{\eps^2}$, $x'=\frac{x}{\eps}$ and $\xi'=\eps\xi$, we find
\begin{eqnarray*}
\sup_{x\in \R^d}\left|\int_{\R^d}e^{it\phie(|\xi|)+ix\cdot\xi}\chi(|\xi|)\,\dd \xi\right|=\sup_{x'\in \R^d}\left|\int_{\R^d}e^{it'\phi(\eps|\xi|)+ix'\cdot\eps\xi}\chi\left({|\xi|}\right)\,\dd \xi\right|\\
=\sup_{x'\in \R^d}\eps^{-d}\left|\int_{\R^d}e^{it'\phi(|\xi'|)+ix'\cdot\xi'}\chi\left(\frac{|\xi'|}{\eps}\right)\,\dd \xi\right|
\leq C\eps^{-d} t'^{\frac{d}{2}}h(\eps r)^{-\frac{1}{2}}=C t^{\frac{d}{2}}h(\eps R)^{-\frac{1}{2}},
\end{eqnarray*}
where we used \eqref{eq:estGNT} in the first inequality. Next, it follows from \eqref{eq:boundh} that
\begin{equation*}
\sup_{x\in \R^d}\left|\int_{\R^d}e^{it\phie(|\xi|)+ix\cdot\xi}\chi(|\xi|)\,\dd \xi\right|\leq \frac{C}{\kappa^{\frac{d}{2}}}t^{-\frac{d}{2}}\left(\frac{\kappa\eps|\xi|}{\sqrt{1+(\eps\kappa)^2|\xi|^2}}\right)^{\delta},
\end{equation*}
for all $\delta\in [0,\frac{d-2}{2}]$.
\end{proof}

We stress that \eqref{eq:disp1} deteriorates for $\kappa$ small. The capillarity coefficient $\kappa$ corresponds to the contribution of the capillarity tensor to the dispersion relation. For $\kappa=0$, the system \eqref{eq:scaledQHD} reduces to the one for classical compressible fluid for which the acoustic dispersion is governed by the wave equation. In this paper, we consider the regime where $\kappa>0$ fixed. For the sake of simplicity, we set $\kappa=1$ in the following.

\subsection{Strichartz estimates}
We recall the Definition for Strichartz admissible exponents. 
\begin{definition}\label{def:betagamma}
The exponents $(p, q)$ are said to be \emph{$\mu$-admissible} if $2\leq q, r\leq\infty$, $(q, r, \mu)\neq(2, \infty, 1)$ and 
\begin{equation*}
\frac2q+\frac{\mu}{r}=\frac{\mu}{2}.
\end{equation*}
We say that a pair is Schr{\"o}dinger or wave admissible if $\mu=\frac{d}{2}$ or $\mu=\frac{d-1}{2}$ respectively.  
Further we introduce $\beta=\beta(r)=\frac12-\frac1r$. 
\end{definition}
The right-hand side of \eqref{eq:disp1} mimics the Fourier multiplier $m(\eps|\xi|)=\frac{\eps|\xi|}{\sqrt{1+\eps^2|\xi|^2}}$ to the power $\mu$, motivating us to define the pseudo-differential operator 
\begin{equation}\label{eq:Ueps}
    \Ueps=\frac{\eps\sqrt{-\Delta}}{\sqrt{1-\eps^2\Delta}}.
\end{equation}
We notice that $|\Ueps|\leq \eps|\nabla|$ as consequence of $m(\eps|\xi|)\leq \eps |\xi|$.
Further, we denote 
\begin{equation}\label{eq:alpha}
    \alpha_0=\frac{d-2}{2}\beta(r), \qquad \alpha_1=\frac{d-2}{2}\left(\beta(r)+\beta(r_1)\right).
\end{equation}
The Strichartz estimates follow from Corollary and Theorem 1.2 in \cite{KT98}.
\begin{lemma}
For $d\geq 2$, $\eps, R>0$ and $\alpha_0, \alpha_1$ as in \eqref{eq:alpha}, let $f\in L^{2}(\R^d)$ and $F\in L^{q'}(0,T;L^{r'}(\R^d))$ such that $\supp(\hat{f}), \supp(\hat F(t))\subset \{\xi\in\R^d\, : \, \frac{1}{2}R\leq |\xi|\leq 2R \}$ Then there exists $C>0$ independent from $T,\eps $ such that for any $(q,r)$, $(q_1,r_1)$ Schr{\"o}dinger admissible pairs and any $\alpha\in[0,\alpha_0]$, it holds
\begin{equation}\label{eq:Strichartz-lh}
\|\eith f\|_{L^q(0,T;L^r(\R^d))}\leq C  \|U_{\eps}^{\alpha}f\|_{L^2(\R^d)}.
\end{equation}
Moreover, for any $\alpha\in[0,\alpha_1]$ one has
\begin{equation}\label{eq:Strichartz-lnh}
\left\|\int_{s<t}e^{i(t-s)\Heps}F(s)\dd s\right\|_{L^q(0,T;L^r(\R^d))}\leq C^2  \left\| U_{\eps}^{\alpha}F\right\|_{L^{q_1'}(0,T;L^{r_1'}(\R^d))}.
\end{equation}
\end{lemma}

\begin{proof}
Given \eqref{eq:disp1} and considering the fact that $\eith$ is an isometry on $L^2(\R^d)$, we observe that Theorem 1.2 of \cite{KT98} yields the desired Strichartz estimates. 
\end{proof}

The next Proposition provides the final Strichartz estimates in Besov spaces. We stress that for $\eps=1$, we recover the Strichartz estimates provided by Theorem 2.1 in \cite{GNT05}. We denote by $\dot{B}_{q,r}^s(\R^d)$ the homogeneous Besov space
, see Chapter 5 of \cite{BL}.

\begin{proposition}\label{prop:Strichartz_A}
Let $d\geq 2$, $\eps>0$. Then there exists $C>0$ independent from $T,\eps $ such that for any $(q,r)$, $(q_1,r_1)$ Schr{\"o}dinger admissible pairs and $\alpha\in [0,\alpha_0]$, one has
\begin{equation}\label{eq:Strichartz-h}
\|\eith f\|_{L^q(0,T;L^r(\R^d))}\leq C \eps^{\alpha}\|f\|_{\dot{H}^{\alpha}(\R^d)},
\end{equation}
and any $\alpha\in[0,\alpha_1]$ it holds
\begin{equation}\label{eq:Strichartz-nh}
\left\|\int_{s<t} e^{i(t-s)\Heps} F(s)\dd s\right\|_{L^q(0,T;L^r(\R^d))}\leq   C \eps^{\alpha} \| F \|_{L^{q_1'}(0,T;\dot{B}_{r_1',2}^{\alpha}(\R^d))}.
\end{equation}
\end{proposition}
We notice that $\alpha_0=\alpha_1=0$ for $d=2$, the estimates \eqref{eq:Strichartz-h} and \eqref{eq:Strichartz-nh} do not provide decay in $\eps$ for $d=2$. In the following, for $R>0$ we denote by $P_{R}$ the frequency cut-off $P_{R}(f)=(\phi_{R}(\xi)\hat{f})^{\vee}$, where $\phi\in C_c^{\infty}$ with support around $|\xi|\sim R$.
\begin{proof}
Localized in frequencies of order $2^r$ inequality \eqref{eq:Strichartz-h} follows from \eqref{eq:Strichartz-lh} upon estimating the operator $\Ueps$ via the bound $m(\eps|\xi|)\leq \eps|\xi|$, see \eqref{eq:Ueps}. It remains to sum over the dyadic blocks,
\begin{eqnarray*}
& \|\eith f\|_{L^q(0,T;\dot{B}_{r,2}(\R^d))}\leq C\left\|\|P_{2^k}\eith f\|_{L^q(0,T;L^r(\R^d))} \right\|_{l_k^2}
\leq C \left\|\|P_{2^k} \Ueps f\|_{L^2(\R^d))} \right\|_{l_k^2}\\
&\leq C\eps^{\alpha}\left\|2^{k \alpha}\|P_{2^r} \Ueps f\|_{L^2(\R^d))} \right\|_{l_k^2}=C\eps^{\alpha}\|f\|_{\dot{B}_{2,2}^{\alpha}(\R^d)}, 
\end{eqnarray*}
for all $\alpha\in[0,\alpha_0]$. We used the Minkowski integral inequality to exchange the order of norms in the first inequality, estimate \eqref{eq:Strichartz-lh} in the second and the bound on $\Ueps$ in the third inequality. Estimate \eqref{eq:Strichartz-nh} follows along the same lines.
\end{proof}

Upon writing the solution of \eqref{eq:symAW} in terms of the Duhamel formula and observing that $\alpha_0\leq \alpha_1$, we infer the desired decay of $(\tseps,\tilde{\Jeps})$ from Proposition \ref{prop:Strichartz_A}.
\begin{corollary}\label{coro:3d}
Let $d\geq 2$ and $\eps>0$, $\alpha_0$ defined by \eqref{eq:alpha}. Then for any Schr{\"o}dinger admissible pairs $(q,r)$, $(q_1,r_1)$ and $\alpha\in[0,\alpha_0]$, one has
\begin{equation*}
    \|(\tseps,\tilde{J}_\eps)\|_{L^q(0,T;L^r(\R^d))}\leq C \eps^{\alpha}\left(\|(\tilde{\sigma}_{\eps}^0,\tilde{J}_\eps^0)\|_{\dot{H}^{\alpha}(\R^d)}+\|\tilde{F}_{\eps}\|_{L^{q_1'}(0,T;\dot{B}_{r_1',2}^{\alpha}(\R^d))}\right).
\end{equation*}
\end{corollary}
We observe that the change of variables is such that control of $(\seps,\Q(\Jeps))$ in terms of the initial data $P_{\leq \eps^{-1}}(\seps^0),P_{\geq \eps^{-1}}(\eps\nabla\seps^0),\Jeps^0$ can be achieved, see \cite{AHM} for details.
\subsection{Dispersion phenomena in 2D}
For $d=2$, estimates \eqref{eq:Strichartz-h} and \eqref{eq:Strichartz-nh} reduce to the classical Strichartz estimates for the Schr{\"o}dinger propagator. More precisely, $\alpha_0=\alpha_1=0$ and Proposition \ref{prop:Strichartz_A} does not provide any decay in $\eps$. However, for frequencies below the threshold $\eps^{-1}$, the propagator $\eith$ is well approximated by $e^{i\frac{t}{\eps}|\nabla|}$. More precisely, $\phie(|\xi|)\sim \frac{|\xi|}{\eps},$
for $|\xi|\lesssim\eps^{-1}$. Based on this observation, in \cite{BDS10} the authors introduce Strichartz estimates that are wave-like for low frequencies and provide decay also for $d=2$ but at the cost of higher loss of regularity compared to Proposition \ref{prop:Strichartz_A}, namely the typical loss of derivatives of the standard Strichartz estimates for the wave equation. By interpolation with \eqref{eq:Strichartz-h}, we achieve decay in $\eps$ for arbitrarily small loss of derivatives.

\begin{proposition}\label{prop:Strichartz_B}
Let $\eps>0$ and $\theta\in[0,1)$. Then, for any $\frac{2-\theta}{2}$-admissible pair $(q,r)$ and $s=3\beta\theta$, it holds
\begin{equation*}
    \|\eith f\|_{L^q(0,T;L^r(\R^2))}\leq C \eps^{\frac{s}{3}}\|f\|_{\dot{H}^{s}(\R^2)}.
\end{equation*}
\end{proposition}
We notice that for $\theta=0$, we recover inequality \eqref{eq:Strichartz-h}. For $\theta=1$, we recover the wave-like estimate for low frequencies. The decay in $\eps$ at high frequencies is achieved by Sobolev embedding. The respective non-homogeneous estimate follows from abstract results provided by \cite{KT98}.
\begin{proof}
Let $\eps>0$. Then, a standard stationary phase argument yields that for all $k\in \Z$ such that $2^k\eps\leq 1$ it holds
\begin{equation*}
   \|P_{2^k}\left(\eith f\right)\|_{L^{\infty}(\R^2)} \leq C (\frac{t}{\eps})^{-\frac{1}{2}}2^{\frac{3k}{2}}\|P_{2^k}f\|_{L^1(\R^2)},
\end{equation*}
By interpolation with \eqref{eq:disp1} we recover for any $\theta\in[0,1]$ that
\begin{equation}\label{eq:disp2}
     \|P_{2^k}\left(\eith f\right)\|_{L^{\infty}(\R^2)} \leq C \eps^{\frac{1}{2}\theta} t^{-\frac{2-\theta}{2}}2^{\frac{3k\theta}{2}}\|P_{2^k}f\|_{L^1(\R^2)},
\end{equation}
provided that $2^k\eps\leq 1$. 
Given that $\eith$ is an isometry on $L^2(\R^2)$ the following Strichartz estimates follows from Theorem 1.2 in \cite{KT98}. Let $(q,r)$ be $\frac{2-\theta}{2}$ admissible then for any $k$ such that $2^k\eps\leq 1$ one has
\begin{equation}\label{eq:low}
    \|P_{2^k}\left(\eith f\right)\|_{L^q(0,T;L^r(\R^2))}\leq C\eps^{\beta\theta}2^{3k\beta\theta}\|P_{2^k}f\|_{L^2(\R^2)},
\end{equation}
where we recall that $\beta=\beta(r)=\frac{1}{2}-\frac{1}{r}$.
For high frequencies, namely $k$ such that $2^k\eps>1$, we observe that if $(q,r)$ is $\frac{2-\theta}{2}$ admissible, then 
\begin{equation*}
    \|P_{2^k}\left(\eith f\right)\|_{L^q(0,T;L^r(\R^2))}\leq 2^{k{\beta\theta}}  \|P_{2^k}\left(\eith f\right)\|_{L^q(0,T;L^{\tilde{r}}(\R^2))},
\end{equation*}
with $(q,\tilde{r})$ Schr{\"o}dinger admissible. For $k$ such that $2^k\eps>1$ it follows from \eqref{eq:Strichartz-lh} that 
\begin{equation}\label{eq:high}
     \|P_{2^k}\left(\eith f\right)\|_{L^q(0,T;L^r(\R^2))}\leq C  2^{k\beta\theta}\|P_{2^k}f\|_{L^2(\R^2)}\leq  C \eps^{2\beta\theta} 2^{3k\beta\theta}\|P_{2^k}f\|_{L^2(\R^2)},
\end{equation}
where we have used that for any any $\alpha>0$ and $k\in \Z$ such that $2^k\eps>1$, one has
\begin{equation}\label{eq:disp3}
   \|P_{2^k}f\|_{L^{2}(\R^2)} \leq C \eps^{\alpha}2^{k\alpha}\|P_{2^k}f\|_{L^2(\R^2)}.
\end{equation}
Finally, for any $\theta\in[0,1]$ and $\frac{2-\theta}{2}$-admissible pair $(q,r)$, it holds
\begin{equation*}
    \|\eith f\|_{L^q\dot{B}_{r,2}^{0}}\leq \|P_{\leq \eps^{-1}}\left(\eith f\right)\|_{L^q\dot{B}_{r,2}^0}+\|P_{\geq \eps^{-1}}\left(\eith f\right)\|_{L^q\dot{B}_{r,2}^0}\leq C \eps^{\beta\theta}\|f\|_{\dot{B}_{2,2}^{3\beta\theta}},
\end{equation*}
where we applied \eqref{eq:low} and \eqref{eq:high} to the low and high frequency part respectively. 
\end{proof}

\section{Applications}\label{sec:app}
In \cite{AHM}, the refined Strichartz estimates of Proposition \ref{prop:Strichartz_A} have been introduced for low Mach number limit analysis of the three-dimensional quantum Navier--Stokes equations that can be seen as a viscous regularization of \eqref{eq:scaledQHD}. More precisely, the momentum equation is augmented by the viscous tensor $2\nu\dive(\reps\mathbf{D} \ueps)$. The result is valid for ill-prepared initial data of finite energy without further smallness or regularity assumptions. 
\begin{theorem}[\cite{AHM}]
{Let $d=3$, $\gamma>1$} and let $(\reps,\ueps)$  be a finite energy weak solution to (QNS) satisfying a Bresch-Desjardins type entropy estimate  with initial data of finite energy and let $0<T<\infty$ be an arbitrary time. Then $\reps-1$ converges strongly to $0$ in $L^{\infty}(0,T;L^{2}(\R^3))\cap L^4(0,T;H^s(\R^3))$ for any $0\leq s<1$. For any subsequence (not relabeled) $\sqrt{\reps}\ueps$ converging weakly to ${u}$ in $L^{\infty}(0,T,L^2(\R^3))$, then $u\in L^{\infty}(0,T;L^2(\R^3))\cap L^2(0,T;\dot{H}^1(\R^3))$ is a global weak solution to the incompressible Navier-Stokes equation with initial data ${u}_{\big| t=0}=\mathbf{P}(u_0)$ and $\sqrt{\reps}\ueps$ converges strongly to $u$ in $L^2(0,T;L_{loc}^2(\R^3))$.\\
Moreover, $\Q(\reps\ueps)$ converges strongly to $0$ in $L^2(0,T,L^q(\R^3))$ for any $2<q<\frac{12}{5}$.
Finally, the limiting solution $u$ also satisfies $u\in L^{\frac{4}{1+4s}-}(0,T;H^s(\R^3))$, for $0\leq s\leq\frac12$.
\end{theorem}
The degeneracy of the viscosity tensor \cite{BDL} leads to a lack of control of the velocity field $\ueps$ making the refined Strichartz estimates crucial for the analysis in \cite{AHM}.
We mention that Proposition \ref{prop:Strichartz_B} provides control of the acoustic dispersion suitable for the low Mach number analysis of (QNS) for $d=2$.\\
In the forthcoming paper \cite{AHM18}, the authors plan to address the $\eps$-limit for system \eqref{eq:en_QHDintro} for $d=2,3$ with far-field behavior \eqref{eq:boundaryQHDintro}. The $\eps$-limit should be interpreted as scaling limit for the healing length $\eps$ going to $0$ and does not coincide with the semi-classical limit. The $\eps$-limit for \eqref{eq:scaledQHD} posed on $\T^d$ for $d=2,3$ has been studied in \cite{DM16}. Due to absence of significant dispersion on periodic domains, more regular solutions to \eqref{eq:scaledQHD} and convergence to local strong solutions is considered. 
\\
We conclude by observing that the acoustic dispersion in a fairly general class of Korteweg fluids \cite{BG,BDL,BGL} is still governed by the Boussinesq type equation \eqref{eq:densityfluct}. The present analysis is hence of interest for the more general framework of Euler- and Navier-Stokes-Korteweg fluids. 

\section*{Acknowledgements}
The first and second author acknowledge partial support through the INdAM-GNAMPA project \emph{Esistenza, limiti singolari e comportamento asintotico per equazioni Eulero/Navier--Stokes--Korteweg}. The second author acknowledges partial support by the Agence Nationale de la Recherche, project SINGFLOWS, grant ANR-18-CE40-0027-01. The third author acknowledges partial support by PRIN-MIUR project 2015YCJY3A\_003 \emph{Hyperbolic Systems of Conservation Laws and Fluid Dynamics: Analysis and Applications}. 

\end{document}